\documentclass[11pt]{amsart}

\usepackage{bm}
\usepackage{fullpage}
\usepackage{amssymb}
\usepackage{amsmath,amsfonts,amsthm}
\usepackage{hyperref}
\usepackage{color}
\usepackage{cite}

\newtheorem{theorem}{Theorem}

\newtheorem{definition}[theorem]{Definition}
\newtheorem{Claim}[theorem]{Claim}

\newtheorem{corollary}[theorem]{Corollary}

\DeclareMathOperator\G{\mathcal{G}}

\DeclareMathOperator\R{\mathbb{R}}

\DeclareMathOperator\F{\mathcal{F}}

\title{Extremal results on $G$-free	colorings of graphs}

\author{Yaser Rowshan$^1$ }
\address{$^1$Y. Rowshan
	Department of Mathematics, Institute for Advanced Studies in Basic Sciences (IASBS), Zanjan 45137-66731, Iran}
\email{y.rowshan@iasbs.ac.ir}
\begin{document}
	\maketitle 
	
	\begin{abstract} 
		Let $H=(V(H),E(H))$ be a graph. A  $k$-coloring of $H$ is a mapping $\pi : V(H) \longrightarrow \{1,2,\ldots, k\}$ so that each color class induces a $K_2$-free subgraph. For a  graph $G$ of order at least $2$, a $G$-free $k$-coloring of $H$
		is  a mapping $\pi : V(H) \longrightarrow \{1,2,\ldots,k\}$ so that 
		the  subgraph of $H$ induced by each color class of $\pi$ is $G$-free, i.e.  contains no  copy of $G$. The $G$-free chromatic number of $H$ is the minimum number $k$ so that  there is a $G$-free $k$-coloring of $H$, denoted by
		$\chi_G(H)$. 	A graph $H$ is uniquely $k$-$G$-free colouring if $\chi_G(H)=k$ and every  $k$-$G$-free colouring of $H$  produces the same color classes. 
		A graph $H$ is minimal with respect to $G$-free, or  $G$-free-minimal, if  for every  edges of $E(H)$ we have $\chi_G(H\setminus\{e\})= \chi_G(H)-1$.
		In this paper we give some  bounds and attribute  about  uniquely $k$-$G$-free colouring and  $k$-$G$-free-minimal. 
		
	\end{abstract}
	
	\section{Introduction} 
	All graphs $G$ considered in this paper are undirected, simple, and finite graphs. For  given graphs $G$, we denote its vertex set, edge set, maximum degree, and minimum degree by $V(G)$, $E(G)$, $\Delta(G)$, and $\delta(G)$, respectively. The number  of vertices of 
	$G$ is define by $|V(G)|$.
	For a vertex $v\in V(G)$, we use $\deg_G{(v)}$ ($\deg{(v)}$) and $N_G(v)$ to denote the degree and neighbors of $v$ in $G$, respectively. 
	
	The join of two graphs $G$ and $H$, define by $G\oplus H$, is a graph obtained from $G$ and $H$ by joining each vertex of $G$ to all vertices of $H$. The union of two graphs $G$ and $H$, define by $G\cup H$, is a graph obtained from $G$ and $H$, where $V(G\cup H) =V(G)\cup V(H)$ and $E(G\cup H) =E(G)\cup E(H)$. For convenience, we use $[n]$  instead of $\{1,2,\ldots,n\}$.
	A $k$-vertex coloring of $H$ is a partition of $V(H)$ into $k$ color classes such that vertices in the same
	class are not adjacent. Moreover, $H$ is called uniquely vertex $k$-colorable if every $k$-coloring of it induces the same partition on $V(H)$. Also, A graph $H$ is minimal with respect to $k$-vertex coloring of $H$, if  for every  edges of $E(H)$ we have $\chi(H\setminus\{e\})= \chi(H)-1$.
	
	Uniquely vertex $k$-colorable have been studied by Chartrand and Geller \cite{chartrand1969uniquely}, Aksionov \cite{aksionov1977uniquely},  Harary, Hedetniemi, and Robinson \cite{harary1969uniquely}, Bollob\'{a}s \cite{bollobas1978extremal, bollobas1978uniquely}, Borwiecki and Burchardt \cite{borowiecki1993classes}, Alishahi and Taherkhani \cite{alishahi2018extremal}, and Xu \cite{shaoji1990size}.
	An $(n, k)$-coloring of a $H$ corresponds to the partition of $V(H)$ into $n$ color classes, so that the
	induces a subgraph with each color class whose maximum degree does not exceed $k$. A graph $H$ is uniquely $(n, k)$-colorable if it is $(n, k)$-colorable and every $(n, k)$-coloring
	of $H$ produces the same color classes.  M. Frick and M. A. Henning show that the following results is true.
	\begin{theorem}\cite{frick1994extremal}\label{1}
		Suppose $H$ be a uniquely $(n, k)$-colorable
		graph, where  $ n\geq  2$ and $k\geq 1$.  Hence:
		\[|V(H)|\geq n (k+1)-1.\]
	\end{theorem} 
	\begin{theorem}\cite{frick1994extremal}\label{2}
		For each  $ n\geq  2$ and $k\geq 1$, there exists a uniquely $(n, k)$-colorable graph with  $n(k + 1)-1$ member.
	\end{theorem} 
	
	\subsection{$G$-free coloring.}	
	The conditional chromatic number $\chi(H,P)$ of $H$,  is the smallest  integer  $k$ for which there 
	is a decomposition of  $V(H)$  into  $k$ color class say $V_1,\ldots, V_k$, so that  $H[V_i]$ satisfies the property $P$, where   $P$, is a graphical property and $H[V_i]$ is a the induced subgraph on $V_i$, for each $1\leq i\leq k$.
	This extension of graph coloring was presented by Harary in 1985~\cite{MR778402}. Suppose that $\G$
	be a families of  graphs, when $P$ is the feature that a subgraph induced by each color class does not contain  each copy of 
	members of $\G$, we write $\chi_{\G}(H)$ instead of $\chi(H, P)$. In this regard, we say a graph $H$ has a $\G$-free $k$-coloring if there is a  map $\pi : V(H) \longrightarrow \{1,2,\ldots,k\}$ such that  each color class  $V_i=\pi^{-1}{(i)}$  does not contain any members of $\G$. For simplicity of notation if $\G=\{G\}$, then we write $\chi_G(H)$ instead of $\chi_{\G}(H)$. 
	
	An ordinary  $k$-coloring of $H$  can be viewed as $\G$-free $k$-coloring of a graph $H$ by taking $\G=\{K_2\}$.  It was shown that for each  graph $H$, $\chi(H)\leq  \Delta(H)+1$. The well-known Brooks theorem states that   for any  connected graph $H$, $\chi(H)\leq  \Delta(H)$ if $H$ is a connected and is neither an odd $C_n$ nor a $K_n$, then $\chi(H)\leq  \Delta$ \cite{Brooks}. One can refer to \cite{chartrand1968generalization, lovasz1966decomposition, harary1985graphs, catlin1995vertex,alishahi2018extremal}  and it references for further studies
	
	A graph $H$ is uniquely $k$-$G$-free colouring if $\chi_G(H)=k$ and every  $k$-$G$-free colouring of $H$  produces the same color classes, and a graph $H$ is minimal with respect to $G$-free, or  $G$-free-minimal, if  for every  edges of $E(H)$ we have $\chi_G(H\setminus\{e\})\leq \chi_G(H)-1$. 
	
	In this article we investigate some properties of  uniquely $k$-$G$-free and $G$-free-minimal coloring of graphs as follow: 
	
	\begin{theorem}\label{A}
		Let $H$ and $G$ be two graph and $k\geq 1$ be a integers. Assume that $H$ is uniquely $k$-$G$-free colourable graph, hence we have:
		\[|V(H)|\geq k |V(G)|-1.\]
	\end{theorem} 
	\begin{theorem}\label{B}
		Let $G$ be a graph with $m$ vertices and $k\geq 1$ be an integers. Then: 
		\begin{itemize} 
			\item[I]: There exists a $k$-$G$-free colourable graph, with $k m-1$ members where $k-1$ class has $m$ member and one class has a $m-1$ member.
			\item[II]: If   $m\geq 2\delta(G)-1$ and $\delta(G)=1$, then there exists a uniquely $k$-$G$-free colourable graph with $km-1$ members and if $\delta \geq 2$ then there exist many $k$-$G$-free colourable of a graph with $km-1$ members.
			\item[III]: If there exists a uniquely $k$-$G$-free colourable graph with $k m-1$ members, then for  each $i\geq 1$ there exists a uniquely $k$-$G$-free colourable graph with $km+i-1$ members.
		\end{itemize}
	\end{theorem} 
	\begin{theorem}\label{C}
		Let $H$ and $G$ be two graph and $k\geq 1$ be a integers. Assume that $H$ is uniquely $k$-$G$-free colourable graph, $|V(G)|=m$, $\delta(G)=\delta$, $|E(H)|=e_1$ and $|E(G)|=e_2$, hence:
		\[e_1\geq  k(e_2-\delta+\frac{k+1}{2}(m\delta)).\]
	\end{theorem} 
	\begin{theorem}\label{D}
		$H$ is $k$-$G$-free-minimal iff either $H$ is a graph $K_{(k-1)(m-1)+1}$ minus the edges of a $1$-factor, when $k$ is even, or $H\cong H_1\oplus K_1$, where $H_1$ is a graph $K_{(k-1)(m-1)}$ minus the edges of a $1$-factor, when $k$ is odd, where $G$ be a such connected graph with $m$-members.
	\end{theorem} 
	\section{Uniquely $G$-free graph}
	\begin{definition}
		A graph $H$ is uniquely $k$-$G$-free colouring if $\chi_G(H)=k$ and every  $k$-$G$-free colouring of $H$  produces the same color classes.
	\end{definition}
	
	In this section we investigate some properties of  uniquely $k$-$G$-free  coloring of graphs and give a lower bounds on the vertices of uniquely $k$-$G$-free  graphs $H$.  In the following theorem  we determine a lower bound for the size of uniquely $k$-$G$-free  coloring of $V(H)$.
	\begin{theorem}\label{t1}
		Let $H$ and $G$ be two graph and $k\geq 1$ be a integers. Assume that $H$ is uniquely $k$-$G$-free colourable graph (u-k-G-f), then:
		\[|V(H)|\geq k |V(G)|-1.\]
	\end{theorem} 
	\begin{proof}
		Assume that $|V(H)|=n$ and $|V(G)|=m$ and let $V_1,V_2,\ldots, V_k$ be a uniquely $k$-partition of $V(H)$ where $H[V_i]$ is a $G$-free class.  Now, considering the following claim.

		\begin{Claim}
			$|V_i|\geq m-1$ for each $i\in[k]$.
		\end{Claim}
		\begin{proof}By contrary, suppose that there exist at least one $i$ say  $i=1$, such that, $|V_1|\leq m-2$. Set $v\in V_2$ and set $V'_1=V_1\cup \{v\}$,  $V'_2=V_2\setminus \{v\}$ and $V'_i=V_i$ for $i\geq3$, hence we have $H[V'_i]$ is $G$-free, a contradiction.
		\end{proof}
		\begin{Claim}
			For at must one $i$, we have $|V_i|=m-1$.
		\end{Claim}
		\begin{proof} By contrary, suppose that   $|V_1|=|V_2|= m-1$. In this case set $v_i\in V_i$ for $i=1,2$, and set $V'_1=(V_1\setminus \{v_1\})\cup \{v_2\}$,  $V'_2=(V_2\setminus \{v_2\})\cup \{v_1\}$ and $V'_i=V_i$ for $i\geq3$, hence we have $H[V'_i]$ is $G$-free, a contradiction again.
		\end{proof}	
		Now by Claim $1,2$ we have $|V(H)|\geq k |V(G)|-1$ and the proof is complete.
	\end{proof}
	In the next theorem we establish that the bounds in Theorem \ref{t1} is best possible. Also we show that if there exists a uniquely $k$-$G$-free colourable graph with $km-1$ members, then for  each $i\geq 1$ there exists a u-k-G-f with $km+i-1$ members.
	\begin{theorem}\label{t2}
		Let $G$ be a graph with $m$ vertices and $k\geq 1$ be an integers. Then: 
		\begin{itemize} 
			\item[I]: There exists a $k$-$G$-free colourable graph, with $k m-1$ members where $k-1$ class has $m$ member and one class has a $m-1$ member.
			\item[II]: If   $m\geq 2\delta(G)-1$ and $\delta(G)=1$, then there exists a u-k-G-f with $km-1$ members and if $\delta \geq 2$ then there exist many $k$-$G$-free colourable of a graph with $km-1$ members.
			\item[III]: If there exists a u-k-G-f with $k m-1$ members, then for  each $i\geq 1$ there exists a u-k-G-f with $km+i-1$ members.
		\end{itemize}
	\end{theorem} 
	\begin{proof}
		Suppose that $H_i$ for $i=1,2,\ldots,k+1$ be a graph, where  $H_i\cong K_{m-1}$ for each $i=1,2,\ldots,k$, and $H_{k+1}\cong K_{k-1}$ where $V(H_{k+1})=\{v_1,v_2,\ldots,v_{k-1}\}$. For each $i=1,2,\ldots,k$ define $H'_i=H_i\cup \{v_i\}$, so that $v_i$ is adjacent to  at most $\delta(G)-1$ vertices of $V(H_i)$, and define $H'_{k}=H_{k}$. Set $H=\oplus_{i=1}^{i=k}H'_i$. Hence, we can say that $|V(H)|=km-1$. Since $|V(H'_i)|=m$ and $|N_{H'_i}(v_i)|\leq \delta-1$, we  can  check that for each $i$, $H'_i$ is $G$-free, that is  $\chi_G(H)\leq k$. Now we have the following claim:
		
		\begin{Claim}\label{cl1}
			$\chi_G(H)=k.$
		\end{Claim}
		\begin{proof}For $k=1$ its clearly, hence assume that $k\geq 2$. By contradiction assume that $\chi_G(H)\leq k-1$, and let $V_1,V_2,\ldots, V_{k-1}$ be a $(k-1)$-partition of $V(H)$ where $H[V_i]$ is a $G$-free class. Since $H[\cup_{i=1}^{i=k}H_i]\cong K_{k(m-1)}$ we have  $\omega(H)\geq k(m-1)$, hence for at least one  $i$ we have $|V_i\cap (\cup_{i=1}^{i=k}H_i)\geq m$ a contradiction. That is $\chi_G(H)\geq k$ and the proof is complete.
		\end{proof}	
		Therefore by Claim \ref{cl1} and by definition $H$ one can say  that the  part (I) is true.	To prove $(II)$, by Claim \ref{cl1}, assume that  $V_1,V_2,\ldots, V_{k}$ be a $k$-partition of $V(H)$, where $H[V_i]$ is a $G$-free for each $i\in[k]$ and $H$ is a graph  constructed  in part I. Hence, it easy to check that  the following claim is true.
		
		\begin{Claim}\label{cl2} 
			$|V_i\cap (\cup_{i=1}^{i=k}H_i)|= m-1$ for each $i\in [k]$.
		\end{Claim}
		Therefore,	by Claim \ref{cl2} we have the  claims as  follow:
		
		\begin{Claim}\label{cl3}
			$|V_i|\leq m$ for $i=1,2,\ldots,k$.
		\end{Claim}
		\begin{proof}By contrary, suppose that  $|V_i|\geq m+1$ for at least one $i$, say $i=1$, and assume that $\{x_1,x_2,\ldots,x_{m+1}\} \subseteq V_1$.  Since $H[V_1]$ is $G$-free, $\chi_G(H)=k$ and by Claim \ref{cl2},  we have $|V_1\cap (\cup_{i=1}^{i=k}H_i)|= m-1$. W.l.g  let $V_1\cap (\cup_{i=1}^{i=k}H_i)=\{x_1,x_2,\ldots,x_{m-1}\}$, that is $x_m,x_{m+1} \in V(H_{k+1})$. W.l.g suppose that $x_m=v_1$ and $x_{m+1}=v_2$. So, since $H[V_i]$ is $G$-free, $|N_H(v_i)\cap \{x_1,x_2,\ldots,x_{m-1}\}|\leq \delta-1$, that is $|\{x_1,x_2,\ldots,x_{m-1}\}\cap V(H_i)|\geq m-\delta$ for $i=1,2$. As $m\geq 2\delta-1$ one can say that there exist at least one $i\in [2]$, say $i=1$ so that $|N_H(v_1)\cap \{x_1,x_2,\ldots,x_{m-1}\}|\geq \delta$, a contradiction to $|N_H(v_i)\cap \{x_1,x_2,\ldots,x_{m-1}\}|\leq \delta-1$, that is the assumption does not hold and the proof of claim is complete.
		\end{proof}		
		Now since $n=km-1$ and by Claim \ref{cl1}, \ref{cl2} and \ref{cl3}, let $|V_i|=m$ for $i=1,2,\ldots,k-1$ and $|V_k|=m-1$. So, suppose that $\delta=1$ and consider the following claim:  
		
		\begin{Claim}\label{c4} 
			The coloring provided is unique. 
		\end{Claim}
		\begin{proof}Consider $V_{i'}$. Since $|V_{i'}|=m$ for each $i'\in[k-1]$ then $|V_{i'}\cap (\cup_{i=1}^{i=k}H_i)|= m-1$ by Claim  \ref{cl2}, and $|V_{i'}\cap V(H_k)|=1$ for each $i'\in[k-1]$, let $v_{i'}\in V_{i'}\cap V(H_k)$ for each $i'\in[k-1]$. If there exists $i,j\in[k]$ where $|V_{i'}\cap V(H_i)|\neq 0\neq |V_{i'}\cap V(H_j)|$, then one can check that  $\deg_{H[V_{i'}]}(v_{i'})\geq 1$, therefore $G\subseteq H[V_{i'}]$, a contradiction. Hence, $V_{i'}=V(H_i)$ for one $i\in[k]$, so, w.l.g suppose that ${i'}=i$ for each $i'\in[k-1]$. Now, as $H[V_{i'}]$ is $G$-free and $\delta=1$, we have $v_{i'}=v_i$, otherwise  $G\subseteq H[V_{i'}]$ a contradiction again. Hence it easy to check that $V_{i'}=V(H'_i)$ for each $i'\in[k]$  and this coloring is unique. 
		\end{proof}
		Therefore by Claim by \ref{c4}, if $\delta(G)=1$ then  the coloring provided is unique. Now, by considering $\delta(G)$ we have the following claim:
		
		\begin{Claim}\label{c5} 
			If,  $\delta(G)\geq 2$, then the coloring provided $H$  is not unique, where $H$ is a graph as define in part I. 
		\end{Claim}
		\begin{proof} As  $\chi_G(H)=k$, suppose that $V_1,V_2\ldots, V_k$ be are coloring of $V(H)$, where, $V_i=V(H_i)$ for each $i=1,2,\ldots,k$, therefore, $H[V_i]$ is $G$-free. Consider $V_i$, for $i=1,2$, as $v_i\in V_i$  and $|N(v_i)\cap V_i|=\delta-1$ for each $i\in \{1,2\}$ and $\delta\geq 2$, w.l.g suppose that $v_i'\in N(v_i)\cap V_i$ for each $i\in \{1,2\}$. Now, considering $V'_1,V'_2\ldots, V'_k$ be are partition of $V(H)$, where $V'_i=V_i$ for each $i=3,4,\ldots,k$, and $V'_1=V_1\setminus \{v'_1\}\cup \{v'_2\}$ and $V'_2=V_2\setminus \{v'_2\}\cup \{v'_1\}$.  Therefore, by definition $H$, one can check that $v_1v'_2\in E(H)$, $v_2v'_1\in E(H)$, and $|N(v_i)\cap V'_i|=\delta-1$ for each $i=\{1,2\}$. Hence, as, $V_i=V'_i$, for each $i\geq3$, and for $i=1,2$,  $|V_i'|=m$, and   $|N(v_i)\cap V'_i|=\delta-1$, we can say that $H[V'_i]$ is $G$-free. Hence, since $V_i\neq V'_i$ for $i=1,2$, we have  the coloring provided $H$  is not unique. 
		\end{proof}
		Therefore by Claim \ref{c5}, if $\delta(G)\geq 2$ then  the coloring provided is not unique and by considering
		any $v\in V_i$ and ,$v'\in V_j$ for each $i,j\in [k]$ we can say there exist many $k$-$G$-free colourable of a graphs $H$ with $km-1$ members.
		
		To prove (III), assume that $H$ be a u-k-G-f  with $k m-1$ members. Suppose that $V_1$, $V_2, \ldots, V_k$ be the partition of $V(H)$, so that $H[V_i]$ is $G$-free  for each $i\in [k]$ and suppose that $|V_k|=m-1$. Therefore, since  $H$ be a u-k-G-f, one can say that $ H[V_i\cup \{v\}]$ contain a copy of $G$ say $G_i$ consist of $v$, for each $v\in V_k$ and each $i\leq k-1$. W.l.g assume that $V'_i=N(v)\cap G_i$.  Set $H'=H\cup\{u\}$, where $u$ is adjacent to all vertices of $V'_i$, for each $i=1,2,\ldots,k-1$. Now, in $H'$, for $i=1,\ldots ,k$ set $W_i$, so that $W_i=V_i$ for each $i\leq k-1$ and $W_k=V_k\cup \{u\}$. As, $W_i=V_i$ for each $i\leq k-1$, $N(u)\cap V_k=\emptyset$ and  $H$ be a u-k-G-f, one can check that, $H'$ be a u-k-G-f graph with $km$ members. By this way it is easy to check that for each $i\geq 1$ there exists a uniquely $k$-$G$-free colourable graph with $km+i-1$ members.
	\end{proof} 
	Suppose that $H$ is uniquely $k$-$G$-free  graph where $|E(H)|=e_1$ and $|E(G)|=e_2$, in the following theorem  we determine a lower bound for the size of  $E(H)$.
	\begin{theorem}\label{t0}
		Let $H$ and $G$ be two graph and $k\geq 1$ be a integers. Suppose that $H$ is u-k-G-f,  where $|V(G)|=m$, $\delta(G)=\delta$, $|E(H)|=e_1$ and $|E(G)|=e_2$, then:
		\[e_1\geq  k(e_2-\delta+\frac{k+1}{2}(m\delta)).\]
	\end{theorem} 
	\begin{proof}
		Suppose that $|V(H)|=n$, $|V(G)|=m$ and let $V_1,V_2,\ldots, V_k$ be a uniquely $k$-partition of $V(H)$ where $H[V_i]$ is a $G$-free class. As $H$ is u-k-G-f, one can say that for any $i\in \{1,2,\ldots,k\}$ and each $v\in V_i$, we have $G\subseteq H[V_j\cup \{v\}]$, for each $j\neq i$, that is  $|E(H[V_i])|\geq e_2-\delta$ and $|N_{V_j}(v)|\geq \delta$. Therefore it is easy to check that, $e_1=\sum_{i=1}^{i=k}|E(H[V_i])+\sum_{i=1}^{i=k-1}e_{i_j}$ where $e_{i_j}=\sum_{j=i+1}^{j=k}|E(H[V_i,V_j])$ for each $i\in \{1,2,\ldots,k\}$. It is easy to say that $e_{i_j}\geq (k-i)m\delta$ for each $i$. Hence, one can check that:
		\[
		e_1 \geq k(e_2-\delta)+\frac{k(k+1)}{2}(m\delta)=k(e_2-\delta+\frac{k+1}{2}(m\delta)).
		\]
		
		Which means that the proof is complete.
	\end{proof}
	In the next results  we give some   attribute  about  uniquely $k$-$G$-free colouring.  It is easy to check that  the following results is true.
	
	\begin{corollary} 
		Let $H$ and $G$ be two graph and $k\geq 1$ be a integers. Assume that $H$ is u-k-G-f  where $|V(G)|=m$, $\delta(G)=\delta$, hence we have:
		\begin{itemize}
			\item  For each $t\leq k-1$, the subgraph induced on the union of any $t$ colour-classes of the unique colouring is an uniquely $t$-$G$-free colourable graph.
			\item  Each vertex $v\in V(H)$ is adjacent with at least $\delta$ vertex in every colour class other
			than the colour class containing $v$, which means that in $H$, $\delta(H)\geq (k-1)\delta$.
		\end{itemize}
		
	\end{corollary}		
	
	\begin{theorem} 
		Let $H$ and $G$ be two graph and $k\geq 1$ be a integers. Assume that $H$ is u-k-G-f  where $|V(G)|=m$, $\delta(G)=\delta$, hence:
		
		If for a vertex $v$ of $H$ we have $\deg(v) = (k-1)\delta$ and the colour class of $v$ contains more than $m$ vertex, then $H\setminus\{v\}$ is also u-k-G-f.
	\end{theorem}
	
	\begin{proof}
		Suppose that $H$ be a u-k-G-f, and assume that $V_1$, $V_2, \ldots, V_k$ be the partition of $V(H)$, so that $H[V_i]$ is $G$-free, $|V_1|=m+1$ and there exist a vertex of $V_1$ say $v$ so that $\deg(v) = (k-1)\delta$. As $H$ be a u-k-G-f, one can say that $|N(v)\cap V_i|=\delta$ for each $i\in \{2,3,\ldots,k\}$. Otherwise if there exist at least one $i$, so that $|N(v)\cap V_i|\leq \delta-1$, then one can say that $H[V_i\cup\{v\}]$ is  a $G$-free class, hence set $V'_1=V_1\setminus\{v\}$, $V'_i=V_i\cup \{v\}$ and $V'_j=V_j$ for each $j\in [k], j\neq 1,i$. Therefore, for each $i\in[k]$, $H[V'_i]$  is $G$-free, a contradiction. Now considering the following claim:

		\begin{Claim}\label{c0}
			$H'=H\setminus\{v\}$ is u-k-G-f.
		\end{Claim}	 
		\begin{proof} 
			As $|V_1|\geq m+1$, and  $H$ be a u-k-G-f, one can say that $H'=H\setminus\{v\}$ is $k$-$G$-free colourable, by considering  $V_1\setminus\{v\}$, $V_2, \ldots, V_k$.  Now, by contrary suppose that the coloring provided $H'$  is not unique. Therefore as $H$ be a u-k-G-f and the coloring provided $H'$  is not unique, then there exist a vertex $u$ of some $V_i(i\geq 2)$ say $u\in V_i$ so that $V'_1=V_1\cup \{u\}$, $V'_i=V_i\setminus\{u\}$ and $V'_j=V_j$ for each $j\geq2$ and $j\neq i$ be a new $k$-$G$-free  color class of $V(H')$. W.l.g we may suppose that $i=2$.  As $H$ be a u-k-G-f, one can say that  $uv\in E(H)$. Now as $uuv\in E(H)$ and $|N(v)\cap V_i|=\delta$ for each $i\in \{2,3,\ldots,k\}$, one can check that $V'_1=(V_1\cup \{u\})\setminus \{v\}$, $V'_2=(V_2\setminus\{u\})\cup \{v\}$ and $V'_j=V_j$ for each $j\geq3$ be a new $k$-$G$-free  color class of $V(H)$, a contradiction.
		\end{proof}
		Hence, by Claim \ref{c0} the proof is complete.
	\end{proof}
	\section{$G$-free minimal}	
	\begin{definition}
		A graph $H$ is vertex-minimal with respect to $G$-free, or  $G$-free-vertex-minimal for short, if  for every  vertex of $V(H)$, $\chi_G(H\setminus\{v\})\leq \chi_G(H)-1$.
	\end{definition}
	
	In this section we investigate some properties of  $G$-free-vertex-minimal (G-f-v-m)  coloring of graphs and give some  lower bounds on the vertices of G-f-v-m graphs $H$. The following theorem establishes a lower bound for the size of G-f-v-m coloring of $H$.
	\begin{theorem}\label{t3}
		Let $H$ and $G$ be two graph and $k\geq 1$ be a integers. Assume that $H$ is k-G-f-v-m graphs, then:
		\[|V(H)|\geq (k-1) (|V(G)|-1) +1.\]
	\end{theorem} 
	\begin{proof}
		Suppose that $|V(H)|=n$, $|V(G)|=m$ and let $V_1,V_2,\ldots, V_k$  denote the partite sets of  $V(H)$, so that $H[V_i]$ for each $i\in[k]$ is  $G$-free and $H[V_i]$ is a maximal $G$-free class of $H\setminus (\cup^{j=i-1}_{j=1} V_j)$ for each $i=i,\ldots,k$. As $\chi_{G}(H)=k$, we can say that $|V_k|\geq 1$. Assume that $v\in V_k$. Therefore by maximality $V_i$, one can say that $H[V_i\cup \{v\}]$ for each $i\leq k-1$, contain a copy of $G$, that is $|V_i|\geq m-1$ for each $1\leq i\leq k-1$.    Hence:
		\[|V(H)|\geq (k-1) (|V(G)|-1) +1.\]
	\end{proof}
	By Theorem  \ref{t3} it is easy to say that the following results is true.
	\begin{theorem}\label{t4}
		Let $H$ and $G$ be two graphs where $|V(H)|= (k-1) (|V(G)|-1) +1$ for some $k\geq 2$. If there exist a subsets of $V(H)$ say $S$, where $|S|\geq |V(G)|$, then:
		\[\chi_G(H)\leq k-1.\]
	\end{theorem} 
	By Theorem \ref{t4}, one can say that $K_{(k-1)(|G|-1)+1}$ be a $k$-$G$-free graphs, for each $G$.
	We shall now construct k-G-f-v-m graphs with order $(k-1) (|V(G)|-1) +1$. Suppose that for $i=1,2,\ldots k-1$,  $H_i$ denote a $K_{(m-2)}$ and $H_k\cong K_{k}$, where $k\leq m-1$. Set $H_1=\oplus_{i-1}^{i=k-1} H_i\cong K_{(k-1)(m-2)}$ and suppose that $V(H_k)=\{v_1,v_2, \ldots ,v_k\}$. Let for each  $i\in \{1,2,\ldots, k-1\}$ denote a $H'_i=H_i\cup \{v_i\}$, where $v_i$ is adjacent to $m-3$ vertices of $V(H_i)$ and $H'_k=\{v_k\}$, say $v_iw^i_{m-2}\notin E(H')$, where $w^i_{m-2}$ be a vertex of $H_i$ for each $i\in \{1,2,\ldots, k-1\}$. Now, set $H^*=\oplus_{i-1}^{i=k} H'_i$. Not that $|V(H^*)|=(k-1)(m-1)+1$ and $\delta(H^*)=(k-1)(m-1)-1$. Assume that  $\F$ denote a family of connected subgraph of $K_m$  with $m$ vertices and minus the
	edges of a $tK_2$ where $t=\frac{m}{2}$. Now, we have the following theorems:
	\begin{theorem}\label{t5}
		For each $G\in \F$,  $H^*$ is $k$-$G$-free-vertex-minimal. 
	\end{theorem} 
	\begin{proof}
		Since  $k\leq m-1$, we can say that  $H_i$ is $G$-free, that is  $\chi_G(H)\leq k$. As $H_1\cong K_{(k-1)(m-2)}\subseteq H^*$,   $|\omega(H^*)|\geq (k-1)(m-2)$ and by denote $H^*$, we can check that $|\omega(H^*)|= (k-1)(m-2)+1$. W.l.g, suppose that $V'=\{x_1,\ldots,x_{(k-1)(m-2)} \}\subseteq V(H^*)$, where $H^*[V']\cong K_{(k-1)(m-2)}$. Now, we have the claims as follow:
		
		\begin{Claim}\label{cl5} For any $S$ subsets of $V(H^*)$ with  $m$ member, $G\subseteq  H^*[S]$.
		\end{Claim}
		\begin{proof} Suppose that $S=\{s_1,\ldots,s_{m} \}$.  Consider $|S\cap V'|$. If $|S\cap V'|\geq m-1$, then one can say that $K_{m}\setminus e\subseteq  H^*[S]$, therefore   $G\subseteq  H^*[S]$. Now, suppose that $|S\cap V'|\leq m-2$.  Assume that $|\omega(H^*[S])|= t$, and w.l.g suppose that $S'=\{s_1,\ldots,s_{t} \}\subseteq S$, where $H^*[S']\cong K_{t}$. Consider $\overline{S'}=S\setminus S'$.  As  $|\omega(H^*[S])|= t$, we can say that $|N_H(s)\cap S'|\leq t-1$ for each $s\in \overline{S'}=\{s_{t+1},\ldots,s_{m} \}$. As, $\delta(H^*)=(k-1)(m-1)-1$ and $H^*[S]\subseteq H^*$, we can say that $|N_H(s)\cap S'|= t-1$ for each $s\in S'$ and $H^*[\overline{S'}]\cong K_{m-t}$, otherwise we can fined a vertex of $S$ say $w$, so that $\deg_{H^*}(w)\leq (k-1)(m-1)-2$, a contradiction. Since $|\omega(H^*[S])|= t$ and $H^*[\overline{S'}]\cong K_{m-t}$, $t\geq m-t$, that is $m\leq 2t$. Therefore we can check that  $H^*[S]\cong (K_m\setminus \frac{m-t}{2}K_2)$. Now, as $G\in \F$ and $t\geq \frac{m}{2}$, we have $G\subseteq K_m\setminus \frac{m}{2}K_2\subseteq H^*[S]$.
		\end{proof}	
		
		Therefore, as $|V(H^*)|=(k-1)(m-1)+1$ and by Claim \ref{cl5} we can say that $\chi_G(H)= k$. Assume that  $v$ be any vertex of $V(H^*)$, set $H''=H^*\setminus\{v\}$. As $|V(H'')|=(k-1)(m-1)$, and $|G|=m$, we can decomposition of $V(H'')$ into $k-1$ class, where each class have  $m-1$ member, that is $\chi_G(H'')\leq k-1$, which means that, $H^*$ is k-G-f-v-m, and the proof of theorem is complete.
	\end{proof}	
	As, $H^*$ is k-G-f-v-m, it is easy to say that, for each graps $G$ with $m$ vertices and  each subgraphs $H$ of $K_{(k-1)(m-1)+1}$, such that $H^*\subseteq H$, we have  	$H$ is k-G-f-v-m.
	
	
	\begin{definition}
		A graph $H$ is minimal with respect to $G$-free, or  $G$-free-minimal, if  for every  edges of $E(H)$ we have $\chi_G(H\setminus\{e\})= \chi_G(H)-1$.
	\end{definition}
	
	In this following results we investigate some properties of $G$-free-minimal  coloring of graphs and give a lower bounds on the vertices of $G$-free-minimal graphs $H$. By argument similar to the proof of Theorem \ref{t3} in the following theorem establishes a lower bound for the size of $G$-free-minimal coloring of $H$.
	\begin{theorem}\label{t6}
		Let $H$ and $G$ be two graph and $k\geq 1$ be a integers. Suppose that $H$ is $k$-$G$-free-minimal graphs, then:
		\[|V(H)|\geq (k-1) (|V(G)|-1) +1.\]
	\end{theorem} 
	By Theorem \ref{t4}, one can say that $K_{(k-1)(m-1)+1}$ be a $k$-$G$-free-minimal graphs, for $G=K_{m}$.
	We shall now construct $k$-$G$-free-minimal graphs with order $(k-1) (|V(G)|-1) +1$ for some graph $G$. Suppose that $\R$  denote a subgraphs of $K_{(k-1)(m-1)+1}$ minus the edges of a $tK_2$, and $G$  denote a subgraphs of $K_{m}$ minus the edges of a $tK_2$, where $t\leq \frac{m}{2}-1$. Now, we have the following theorems:
	\begin{theorem}\label{t5}
		$\R$  is $k$-$G$-free-minimal.
	\end{theorem} 
	\begin{proof}
		As  $|\R|=(k-1)(m-1)+1$, we can say that $\chi_G(\R)\leq k$. As, $\R$ denote a subgraphs of $K_{(k-1)(m-1)+1}$ minus the edges of $tK_2$, assume that $V(\R)=\{v_1, v_2, \ldots, v_{(k-1)(m-1)+1}\}$ and w.l.g suppose that $E(tK_2)=\{e_1, e_2, \ldots, e_t\}$, where for each $i\in \{1,2,\ldots,t\}$, $e_i=v_iv_{\frac{m}{2}+i}$. Set $V_1=\{v_1, \ldots, v_t\}$,  $V_2=\{v_{\frac{m}{2}+1},\ldots, v_{\frac{m}{2}+t}\}$ and $V_3=V(\R)\setminus (V_1\cup V_2)$. Now, to prove $\chi_G(\R)=k$, by Claim \ref{cl5} we can check that the following claims is true:

		\begin{Claim}\label{cl6} For any $S$ subsets of $V(\R)$ with  $m$ member, $G\subseteq  \R[S]$.
		\end{Claim}
		Therefore, as $|\R|=(k-1)(m-1)+1$  by Claim \ref{cl6} we can say that $\chi_G(\R)=k$. Now, as $|V(\R)|=(k-1)(m-1)+1$, assume that  $v$ be any vertex of $V(\R)$, set $\R'=\R\setminus\{v\}$. As $|V(\R')|=(k-1)(m-1)$, and $|G|=m$, we can decomposition of $V(\R')$ into $k-1$ class, where each class have  $m-1$ member, that is $\chi_G(\R')= k-1$, which means that, $\R$ is $k$-$G$-free-vertex-minimal. Hence, to prove  $\R$ is $k$-$G$-free-minimal, we most show that for any edges of $E(\R)$, say $e$,  $\chi_G(\R\setminus e)= k-1$.  Now, to prove  $\chi_G(\R\setminus e)= k-1$, we have the following claim:

		\begin{Claim}\label{cl7} For any edges of $E(\R)$, say $e$,  $\chi_G(\R\setminus e)= k-1$.
		\end{Claim}	
		\begin{proof}
			Suppose that $e=vv'\in E(\R)$, now by considering $v$ and $v'$, we have three cases as follow:
			
			{\bf Case 1:} $v,v'\in V_i$ for one $i, i\in \{1,2\}$. W.l.g we may suppose that $v,v'\in V_1$. Set $S=V_1\cup V_2$, one can say that $|S|\geq m$, Now, as $v,v'\in S$ and $vv'\in E(\R)$, we can check that  $|N(x)\cap S|\leq d-3$, for each $x\in \{v,v'\}$, therefore $\R[S]$ is $G$-free. Now as, $|V(H)|=(k-1)(m-1)+1$, $|V(G)|=m$ and $|S|\geq m$, one can say that $\chi_G(\R\setminus e)= k-1$.
			
			{\bf Case 2:} $v\in V_i$  and $v'\in V_j$, where $i\neq j$ and  $i, j\in \{1,2\}$. The proof is same as Case 1.
			
			{\bf Case 1:} $v,v\in V_3$. Set $S=V_1\cup V_2\cup \{v,v'\}$, one can say that $|S|\geq m$, Now, as $v,v'\in S$ and $vv'\in E(\R)$, we can check that  $\R[S]\cong K_m\setminus \frac{m}{2}K_2$, therefore as  $t\leq \frac{m}{2}-1$, $\R[S]$ is $G$-free. Now as, $|V(H)|=(k-1)(m-1)+1$, $|V(G)|=m$ and $|S|\geq m$, one can say that $\chi_G(\R\setminus e)= k-1$.
			
			Therefore by Cases 1,2,3 we have the claim is true, which means that the proof is complete.		
		\end{proof}	
	\end{proof}
	
	Suppose that $H$ is a connected subgraphs of $K_{(k-1)(m-1)+1}$, and $G$  denote a subgraphs of $K_{m}$ minus the edges of a $ \frac{m}{2}K_2$, where  $m$ is even and $\chi_G(H)=k$. Now, we have the following theorems:
	\begin{theorem}\label{t5}
		$H$ is $k$-$G$-free-minimal iff either $H$ is a graph $K_{(k-1)(m-1)+1}$ minus the edges of a $1$-factor, when $k$ is even, or $H\cong H_1\oplus K_1$, where $H_1$ is a graph $K_{(k-1)(m-1)}$ minus the edges of a $1$-factor, when $k$ is odd.
	\end{theorem} 
	\begin{proof}
		Suppose that, either $H$ is a graph $K_{(k-1)(m-1)+1}$ minus the edges of a $1$-factor, when $k$ is even, or $H\cong H_1\oplus K_1$, where $H_1$ is a graph $K_{(k-1)(m-1)}$ minus the edges of a $1$-factor, when $k$ is odd. As   $\chi_G(H)=k$, suppose that $e$ be a arbitrary edges of $E(H)$. Therefore, by definition $H$, one can check that $\delta(H\setminus e)=(k-1)(m-1)-3$, that is there exists at least one vertex of $V(H)$, say $v$, so that $\deg(v)=(k-1)(m-1)-3$, and w.l.g suppose that $v',v'' \notin N(v)$. Now, suppose that $S$, be a subsete of $V(H)$ with $m$, member, where $v,v',v''\in S$. As  $v,v',v''\in S$, and  $G$  denote a subgraphs of $K_{m}$ minus the edges of a $\frac{m}{2}K_2$, one can check that $H[S]$ is a $G$-free.  Therefore, since $|V(H)|=(k-1)(m-1)+1$, $|V(G)|=m$,  $|S|=m$, and $H[S]$ is a $G$-free, it is easy to say that $\chi_G(H\setminus e)=k-1$, which means that $H$ is $k$-$G$-free-minimal.
		
		Suppose that $H$ is $k$-$G$-free-minimal, and $k$ is even. As 	$H$ is $k$-$G$-free-minimal and $m,k$ is even, one can say that  $\delta(H)\geq (k-1)(m-1)-1$. Therefore, as $|V(H)|=(k-1)(m-1)+1$ and 	$\delta(H)\geq (k-1)(m-1)-1$, one can check that $K_{(k-1)(m-1)+1}$ minus the edges of a $1$-factor be a subgraph of $H$. Now, assume that  $H'\subset H$, where $H'$ is a graph $K_{(k-1)(m-1)+1}$ minus the edges of a $1$-factor. So there exist at least two vertices of $V(H)$ say $v,v'$, so that $\delta(x)= (k-1)(m-1)$, for each $x\in \{v,v'\}$, that is $vv'\in (H)$. Consider $e=vv'$, as  $H'\subset H$ and $e\notin E(H')$,  we have $H'\subseteq H\setminus e$. Which means that  $k=\chi_G(H')\leq \chi_G(H\setminus e)= k-1$, a contradiction. For the case that $k$ is add, as  $k$ is add and $m$ is even, we have $(k-1)(m-1)+1$ is add. As	$H$ is $k$-$G$-free-minimal we can say that  $\delta(H)\geq (k-1)(m-1)-1$. Therefore, as, $|V(H)|=(k-1)(m-1)+1$, $\delta(H)\geq (k-1)(m-1)-1$ and $(k-1)(m-1)+1$ is add, one can check that $K_{(k-1)(m-1)+1}$ minus the edges of a $tK_2$ be a subgraph of $H$, where $t=\frac{(k-1)(m-1)}{2}$. Hence, there exist a vertex of $V(H)$ say $v$ so that $\deg(v)=(k-1)(m-1)$.
		
		Now, assume that  $H'\subset H$, where $H'$ is a graph $K_{(k-1)(m-1)+1}$ minus the edges of $t'K_2$ be a subgraph of $H'$, where $t'=\frac{(k-1)(m-1)}{2}-1$. So there exist at least three vertices of $V(H)$ say $v,v',v''$, so that $\delta(x)= (k-1)(m-1)$, for each $x\in \{v,v',v''\}$, that is $v'v''\in (H)$. Consider $e=v'v''$, as  $H'\subset H$ and $e\notin E(H')$,  we have $H'\subseteq H\setminus e$. Which means that  $k=\chi_G(H')\leq \chi_G(H\setminus e)= k-1$, a contradiction to $H$ is $k$-$G$-free-minimal.
		
		Therefore in any case if $H$ is $k$-$G$-free-minimal then, either $H$ is a graph $K_{(k-1)(m-1)+1}$ minus the edges of a $1$-factor, when $k$ is even, or $H\cong H_1\oplus K_1$, where $H_1$ is a graph $K_{(k-1)(m-1)}$ minus the edges of a $1$-factor, when $k$ is odd, which means that the proof is complete.
		
	\end{proof}
	
	\bibliographystyle{plain}
	\bibliography{G-free3}
\end{document}